   \documentclass[preprint,3p,11pt]{elsarticle} 

\usepackage{amssymb}
\usepackage{amsmath}
\usepackage{amsthm}

\theoremstyle{plain}
\newtheorem{thm}{Theorem}
\newtheorem{cor}{Corollary}

\theoremstyle{definition}
\newtheorem{defn}{Definition}[section]
\newtheorem{prop}{Proposition}
\newtheorem{exam}{Example}[section]

\theoremstyle{remark}
\newtheorem*{remarks}{Remarks}
\newtheorem*{note}{Note}

\usepackage{varioref}
\labelformat{section}{Section~#1}
\labelformat{subsection}{Section~#1}
\labelformat{figure}{Figure~#1}
\labelformat{table}{Table~#1}
\labelformat{thm}{Theorem~#1}
\labelformat{prop}{Proposition~#1}

\usepackage{tabls}
\usepackage{mathtools,algorithm,algpseudocode}

\DeclareMathOperator\poi{Poiss}
\DeclareMathOperator\var{Var}

\DeclareMathOperator\tv{TV}

\newcommand\cala{\mathcal{A}}

\newcommand\caln{\mathcal{N}}

\newcommand\calx{\mathcal{X}}

\newcommand\bbn{\mathbb{N}}

\newcommand\Pto{\overset{P}{\longrightarrow}}
\newcommand\pto{\overset{P}{\to}}
\newcommand{\set}[1]{\left\{ #1 \right\}}

\usepackage{color,kbordermatrix}

\journal{Advances in Applied Mathematics}

\begin{document}

\begin{frontmatter}



\title{Poisson approximation for large permutation groups}

 \author[label1]{Persi Diaconis}
 \affiliation[label1]{organization={Departments of Mathematics and Statistics, Stanford University},
             addressline={390 Jane Stanford Way},
             city={Stanford},
             postcode={94305},
             state={California},
             country={USA}}

\author[label2]{Nathan Tung}
 \affiliation[label2]{organization={Stanford University},
             addressline={390 Jane Stanford Way},
             city={Stanford},
             postcode={94305},
             state={California},
             country={USA}}





%

\begin{abstract}
Let $G_{k,n}$ be a group of permutations of $kn$ objects which permutes things independently in disjoint blocks of size $k$ and then permutes the blocks. We investigate the probabilistic and enumerative aspects of random elements of $G_{k,n}$. This includes novel limit theorems for cycles of various lengths, number of cycles, and inversions. The limits include compound Poisson distributions with interesting dependence structure.
\end{abstract}




\end{frontmatter}



\section{Introduction}\label{sec1}

Probabilistic and enumerative theory for random permutations is a classical subject. Pick $\sigma\in S_n$ uniformly, what does it look like? A permutation splits into disjoint cycles, i.e., fixed points, transpositions, etc. Write $\sigma\sim\prod_{i=1}^ni^{a_i(\sigma)}$ if $\sigma$ has $a_i(\sigma)$ cycles of length $i$. Following Monmort (1708), \cite{gonch42,gonch44}, \cite{shepp}, roughly, for large $n$, the $\{a_i\}$ have limiting independent Poisson distributions $\{A_i\}$ with $A_i\sim\poi(1/i)$. Many refinements, extensions and applications can be found in the survey paper of \cite{fulman}, the book of \cite{arratia03}, and the recent survey \cite{pd2022}.

This paper develops similar theory for large subgroups of $S_n$. Let $\Gamma\subseteq S_k$, $H\subseteq S_n$, and let
\begin{equation}
\Gamma^n\rtimes H =\left\{(\gamma_1,\gamma_2,\dots,\gamma_n;h)\right\},\qquad \gamma_i\in\Gamma,\ h\in H,
\label{11}
\end{equation}
be the Wreath product. This permutes $\{1,\dots,kn\}$ with $\gamma_1$ permuting $\{1,\dots,k\}$, $\gamma_2$ permuting $\{k+1,\dots,2k\}$, $\dots,\gamma_n$ permuting $\{(n-1)k+1,\dots,kn\}$ independently and then $h$ permuting those $n$ blocks. For example, $(k=2,n=3)$: $((12),(1)(2),(12);(312))$ permutes 1 2 3 4 5 6 first to 2 1 3 4 6 5 and then 6 5 2 1 3 4. In the usual two-line notation for permutations,
\begin{center}\begin{tabular}{cccccc}
1&2&3&4&5&6\\
6&5&2&1&3&4
\end{tabular}\end{center}
has cycle decomposition $(164)(253)$.

Widely occurring examples of Wreath products include:
\begin{itemize}
\item $B_n=C_2^n\rtimes S_n$: the symmetry group of the hypercube. This can also be seen as the group of centrally symmetric permutations in $S_{2n}$.
\item $C_k^n\rtimes S_n$: generalized permutation group
\item $S_k^n\rtimes S_n$: a maximal subgroup (O'nan--Scott theorem)
\end{itemize}

\noindent More details, references and applications are in \ref{sec21} and \ref{sec22}.

Here $\Gamma^n\rtimes H\subseteq S_{kn}$ and $|\Gamma^n\rtimes H|=|\Gamma|^n |H|$. We propose picking $\sigma\in\Gamma^n\rtimes H$ at random and then ask: what does it look like? for $\sigma\in\Gamma^n\rtimes H$, suppose $\sigma$ has $a_i$ $i$-cycles, so $\sum_{i=1}^{kn}ia_i(\sigma)=kn$. Write $\sigma\sim\prod_{i=1}^{kn}i^{a_i(\sigma)}$. Here are two special cases of our main theorem.

\begin{exam}[$S_3^n\rtimes S_n$]

\begin{thm}
Pick $\sigma\in S_3^n\rtimes S_n$ uniformly. Write $\sigma\sim\prod_{i=1}^{3n}i^{a_i(\sigma)}$. Then $\{a_i(\sigma)\}_{i=1}^{3n}$ converges in distribution to $\{A_i\}_{i=1}^\infty$ with joint distribution
\begin{equation*}\begin{array}{ll}
i\equiv1(6)&A_i=3W_i+Z_i\\
i\equiv2(6)&A_i=3W_i+Z_i+Z_{i/2}\\
i\equiv3(6)&A_i=3W_i+Z_i+Y_i\\
i\equiv4(6)&A_i=3W_i+Z_i+Z_{i/2}\\
i\equiv5(6)&A_i=3W_i+Z_i\\
i\equiv0(6)&A_i=3W_i+Z_i+Z_{i/2}+Y_i\\
\end{array}\end{equation*}
where
\begin{equation*}
W_i\sim\poi(1/6i),\quad Z_i\sim\poi(1/2i),\quad Y_i\sim\poi(1/i),
\end{equation*}
and all $W_i$, $Z_i$, $Y_i$ are independent.
\label{thm11}
\end{thm}

\begin{remarks}
\ 
\begin{enumerate}
\item All the $A_i$ have marginal compound Poisson distributions (see \ref{sec23} for definition).
\item The $A_i$ are dependent. For example, $A_1=3W_1+Z_1$, $A_2=3W_2+Z_2+Z_1$ share a common component $Z_1$. Similarly, $A_j$ and $A_{2j}$ are dependent (but $A_j$ and $A_{4j}$ are independent).
\item Many $A_i$ are independent; $A_1,A_3,A_5,A_7,\dots$ are independent, as are $A_j,A_{j+1},\dots$, $A_{2j-1}$ for any $j$, $1\leq j<\infty$.
\end{enumerate}
\end{remarks}
\end{exam}

\begin{exam}[$C_k^n\rtimes S_n$]
Here the permutations within blocks are only by cyclic shifts, independently for each block. In this case, the $\{A_i\}_{i=1}^\infty$ are independent. Note that when $k=1$ the following result specializes to the classical Poisson limit for the cycles of a random permutation.
\begin{thm}
For fixed $k$ and $n$ large, $\sigma\in C_k^n\rtimes S_n$ has $\{a_i(\sigma)\}_{i=1}^{kn}$ converging in distribution to $\{A_i\}_{i=1}^\infty$ with
\begin{equation*}
A_i=\sum_{l\mid(i,k)}\frac{k}{l}Y_{i,l},\qquad Y_{i,l}\sim\poi\left(\frac{l\phi(l)}{ki}\right)
\end{equation*}
where $(i,k)$ is the greatest common divisor, $\phi(l)$ is Euler's totient function, and $Y_{i,l}$ are independent for all $i,l$. Note that in this case the $A_i$ are all mutually independent.
\label{thm12}
\end{thm}
\end{exam}

\ref{sec2} gives background on Wreath products (\ref{sec21}), motivation from a new algorithm for generating random partitions of $n$, i.e., the commuting graph process (\ref{sec22}), background on compound Poisson distributions (\ref{sec23}), and background on cycle indices and P\'olya theory (\ref{sec24}). This last is illustrated by a normal limit theorem for the number of cycles of $\sigma$ in \ref{sec25}. A general version of \ref{thm11} and \ref{thm12} appears in \ref{sec3}. These theorems are proved using P\'olya's cycle index theory, Poissonization and de-Poissonization. A coupling proof which gives rates of convergence is in \ref{sec4}. The final \ref{sec5} gives limit theory for inversions, descents, and a different action of $\Gamma\times H$. 

\section{Background}\label{sec2}

This section gives background on Wreath products (\ref{sec21}), the commuting graph process (\ref{sec22}), compound Poisson distributions (\ref{sec23}), and cycle indices and P\'olya theory (\ref{sec24}). As an application, it derives a central limit theorem for the number of cycles in a random element of $\Gamma^n\rtimes H$ using cycle theory and Anscomb's central limit theorem for stopped sums.

\subsection{Wreath products}\label{sec21}

For $\Gamma\subseteq S_k$, $H\subseteq S_n$, $\Gamma^n \rtimes H$ is the Wreath product sometimes denoted $\Gamma \, W_r \, H$. It occurs frequently in basic group theory with its own Wikipedia page. Any text on group theory has a section on Wreath products (for example \cite{suzuki1982group} Chapter 2, Section 10). They often appear as ``troublemakers'', providing counter-examples. For instance, $C_p \, W_r \, C_p$, a group of order $p^{p+1}$, is the Sylow-$p$ subgroup of the symmetric group $S_{p^2}$. It is the smallest example of a non-regular $p$-group. A spirited literature review of wreath products and a host of applications in number theory appear in \cite{barquero2024embedding}. Other applications include:
\begin{description}
\item $B_n=C_2^n\rtimes S_n$, the hyperoctahedral group, is the group of symmetries of an $n$-dimensional hypercube. This may also be seen as the subgroup of $S_{2n}$ of centrally symmetric permutations $\sigma(2n-i+1)+\sigma(i)=2n+1$. For example in $S_{10}$:
\begin{equation*}
\begin{matrix}1&2&3&4&5&6&7&8&9&10\\8&1&5&4&9&2&7&6&10&3\end{matrix}.
\end{equation*}
These occur as the arrangements of $2n$ cards obtainable with the two types of perfect shuffles \cite{pd1983}. In \cite{pd1998} it is shown that these permutations biject with phylogenetic trees and their cycle theory is useful for understanding Markov chains on trees. $B_n$ is also the Weyl group of the symplectic group $Sp(2n,q)$ and \cite{fg17} needed results about fixed points in $B_n$ for their work on this group.

\item $C_k^n\rtimes S_n$ is the generalized permutation group. This may be pictured as $n\times n$ permutation matrices with the usual ones replaced by $k$th roots of unity. It has its own Wikipedia page. We encounter it in understanding what permutations commute with a given $\sigma\in S_n$. If $\sigma\sim\prod_{i=1}^ni^{a_i(\sigma)}$, then the subgroup of permutations commuting with $\sigma$, $C_{S_n}(\sigma)=\prod_{i=1}^nC_i^{a_i}\rtimes S_{a_i}$. Further explanation is in \ref{sec22}. Of course, $C_2^n\rtimes S_n\cong B_n$.

\item $S_k^n\rtimes S_n$ occurs in statistical work. With $n$ groups of people, each with $k$ subjects, a natural symmetry for permutation tests of group homogeneity is $S_k^n\rtimes S_n$. See \cite{bailey} for a textbook account. It also occurs in group theory: the O'nan--Scott theorem, \cite{dixon} for example, classifies maximal subgroups of $S_N$, i.e., equivalently primitive actions of $S_N$. There is a small list of possibilities and, for $N=kn$, $S_k^n\rtimes S_n$ is on the list. Properties of cycles are used to prove theorems about the distribution of fixed points and derangements in \cite{pd2008}.
\end{description}

Iterated Wreath products occur too. The Sylow-$p$ subgroup of $S_{p^n}$ is $C_p \, W_r \, C_p \, W_r \, \cdots  \, W_r \, C_p$. There has been substantial probabilistic and enumerative effort at understanding the conjugacy structure of these groups. See \cite{abert, palfy1989further} and their references.

In this paper we study the action of Wreath products as subgroups of $S_{kn}$. It is also natural to study the conjugacy classes of $\Gamma^n\rtimes H$ in similar vein. These are clearly identified in \cite{bernhard,james,macd} but we have not seen probabilistic investigation.

\subsection{The commuting graph process and random partitions}\label{sec22}

Our motivation for the present study arose from analysis of an algorithm for generating uniformly distributed partitions of $n$. More generally, let $G$ be a finite group and consider the \textit{commuting graph of $G$}. This has vertices indexed by $G$ and an edge from $s$ to $t$ if $st=ts$ in $G$. Let $K(s,t)$ be the transition matrix of the natural random walk on $G$. So $K(s,t)=|C_G(s)|^{-1}$ if $s$ and $t$ commute, zero otherwise. The commuting graph walk was suggested by Peter Cameron in a survey paper \cite{cameron2021graphs} reviewing graphs defined via groups. Important advances are in \cite{giudici2013there, giudici2016realizability}.

\begin{prop}
$K(s,t)$ is a reversible Markov chain with unique stationary distribution $\pi(s) \propto |\kappa(s)|^{-1}$, $\kappa(s)$ the conjugacy class of $s$. Further, the chain generated by $K$ lumps to an ergodic, symmetric Markov chain on conjugacy classes with transition matrix
\begin{equation*}
P(C,C')=\frac1{|C_G(C)|}|C_G(C)\cap C'|.
\end{equation*}
\end{prop}

The proof is elementary Markov chain theory; see \cite{cameron}. The point is that this gives a method to generate a uniformly chosen conjugacy class. As an example, consider $G=S_n$; the conjugacy classes are indexed by partitions of $n$. The Markov chain is easy to run: from $\sigma\in S_n$, choose $\eta\in C_{S_n}(\sigma)$ uniformly and move to $\eta$. When $n=5$, the transition matrix of the lumped chain on partitions becomes the symmetric matrix:
\begin{equation*}
\frac1{120}\quad \kbordermatrix{
&1^5&1^32&1^23&14&12^2&23&5\\
1^5 &1 &10&20&30&15&20&24\\
1^32&10&40&20&0 &30&20&0 \\
1^23&20&20&40&0 &0 &40&0 \\
14  &30&0 &0 &60&30&0 &0 \\
12^2&15&30&0 &30&45&0 &0 \\
23  &20&20&40&0 &0 &40&0 \\
5   &24&0 &0 &0 &0 &0 &96}
\end{equation*}

For general $n$, the entries can be given analytically using cycle indices. In the notation of \ref{sec24}, if $\lambda \sim \prod_{i=1}^n i^{a_i}, \lambda' \sim \prod_{i=1}^n i^{a_i'}$
$$
K(\lambda, \lambda') = [x_1^{a_1'}x_2^{a_2'}\dots x_n^{a_n'}] \prod_{i=1}^n Z_{C_i^{a_i} \rtimes S_{a_i}}
$$
where $[x]P(x)$ represents the coefficient of $x$ in the polynomial $P$. For the first row this gives
$$
K(1^n,\lambda')=(z_{\lambda'})^{-1}=K(\lambda',1^n),\qquad z_{\lambda'}=\prod_{i=1}^ni^{a_i(\lambda')}a_i!.
$$
which is simply the chance that a random permutation has cycle type $\lambda'$. Also 
$$
K((n),\lambda')=\begin{cases}\phi(d)/n&\text{if }\lambda'=d^{n/d}\text{ for }d|n\\0&\text{otherwise.}\end{cases}=K(\lambda',(n))
$$
Thus when $n=p$, a prime, the last row $K((p),\lambda')=0$ except when $\lambda'=1^p$ or $p$, as in the example. 

For $\lambda=\lambda_1>\lambda_2>\dots>\lambda_l$ (all parts distinct), $K(\lambda,\cdot)$ can be described as: for each $\lambda_i$ pick a partition as the cycle type of a uniform permutation from $C_{\lambda_i}$ independently and take $\lambda'$ as the union of these partitions. Thus in one step each piece of $\lambda$ may only split into parts of equal size, it's only when $\lambda$ contains multiple pieces of the same length that merging can occur.

For $\sigma$ of type $\lambda$, $C_{S_n}(\sigma)=\prod_{i=1}^n(C_i^{a_i}\rtimes S_{a_i})$ shows that the distribution of the cycles of $C_i^a\rtimes S_a$ determine the transition matrix. We needed these to understand Doeblin bounds and couplings of the chain $K(\lambda,\lambda')$. This is ongoing work (!).

There are other algorithms for generating random partitions --- see \cite{arratia16} --- but numerical experimentation suggests the commuting graph chain. Practical implementation of the lumped commuting graph process is carried out in \cite{diaconishowes}, where it was critical to use the lumped chain. The algorithm works easily for $n \approx 10^9$, where it seems to converge after around 30 steps. It is also a close cousin of hit-and-run and Swendsen--Wang-type block spin algorithms. See \cite{pd2007} which further motivates analysis.

\subsection{Compound Poisson distributions}\label{sec23}

As the examples in the Introduction show, the natural limit laws in this part of the world often involve compound Poisson laws: nonnegative integer combinations of independent Poisson variates. This is a well-studied family in probability that also spans other disciplines; for a surprising appearance in ergodic theory see \cite{bahsoun2024rare} and its wealth of references. A survey of examples and techniques for proving limit theorems by Stein's method is in \cite{barbour}. This offers a different way to prove our theorems.

\begin{defn}
Let $X_1,X_2,\dots$ be i.i.d.\ with $P(X_i=j)=\theta_j\geq0$ for $j=1,2,\dots$. Let $N\sim\poi(\lambda)$. Then
\begin{equation}
W=\sum_{i=1}^NX_i
\label{21}
\end{equation}
has a \textit{compound Poisson law} with parameters $\{\theta_j\}_{j=1}^\infty$ and $\lambda$. Patently, compound Poisson laws are infinitely divisible and a classical theorem of \cite{feller} shows that these are all the infinitely divisible laws supported on $\bbn$.
\end{defn}

An equivalent definition can be shown: let $Y_i$ be independent, $\poi(\theta_i\lambda)$. Then
\begin{equation}
W=\sum_{j=1}^\infty jY_j
\label{22}
\end{equation}
has a compound Poisson law as in \eqref{21}.

In present applications, collections of dependent variates with compound Poisson margins occur. Our examples match up with the standard framework for constructing dependent, infinitely divisible vectors. See \cite{dwass} or \cite{letac}. A clean definition of compound Poisson vectors appears in \cite{ellis}:

\begin{defn}
Fix $\{\theta_j\}_{j=1}^\infty$ summing to $1$. For each $\lambda\geq0$, let $Q_\lambda$ be the measure defined by \eqref{21}. Fix a finite or countable index set $A$ and let $\cala$ be the set of non-empty subsets of $A$. Finally, let $\Lambda:\cala\to[0,\infty)$ be a function with $\sum_{S\in\cala}\Lambda(S)<\infty$. Define, for $a\in A$,
\begin{equation*}
W_a=\sum_{a \in S \in \cala}Y_S
\end{equation*}
where $Y_S$ are independent compound Poisson with parameters $\{\theta_j\}$ and $\Lambda(S)$. Then $\{W_a\}_{a\in A}$ is a compound Poisson random vector.
\end{defn}

\begin{note}
The same component $Y_S$ may appear in several different $W_a$ so the $W_a$ are (generally) dependent. The vector $\{W_a\}$ is infinitely divisible (and all margins are as well). We do not know if this class includes all infinitely divisible vectors supported on $\bbn^A$.
\end{note}

\subsection{Cycle indices and P\'olya theory}\label{sec24}

P\'olya theory studies enumeration under symmetry. Let $G$ be a finite group operating on a finite set $\calx$. This splits $\calx$ into disjoint orbits
\begin{equation*}
\calx=O_1\cup O_2\cup\dots\cup O_n\qquad\text{with }x\in O_x\text{, say.}
\end{equation*}
Natural questions include:
\begin{itemize}
\item How many orbits are there?
\item What are the orbit sizes?
\item Are there ``natural'' or ``nice'' labels for orbits?
\item How can an orbit be chosen uniformly?
\end{itemize}

The commuting graph process of \ref{sec22} provides an example. There $\calx=G$ and $G$ acts on itself by conjugation. The Goldberg--Jerrum \cite{jerrum} Burnside process is similar.

Suppose now that $G\subseteq S_n$ is a permutation group. The \textit{cycle index polynomial}
\begin{equation}
Z_G(x_1,x_2,\dots,x_n)=\frac1{|G|}\sum_{s\in G}\prod_{i=1}^nx_i^{a_i(s)}
\label{23}
\end{equation}
is a useful tool for P\'olya-type problems; \cite{constant,debru,james,polya, bergeron1998combinatorial} are good sources for this connection. They contain the following basic facts, used below, with $C_n$ the cyclic group on $[n]$ and $\phi$ Euler's totient function:
\begin{enumerate}[(i)]

\item $Z_{C_n}(x_1,\dots,x_n)=\frac1{n}\sum_{d|n}\phi(d)x_d^{n/d}$\label{24}

\item
$Z_{S_n}(x_1,\dots,x_n)=\sum_{\lambda \vdash n}\frac1{z_\lambda}\prod_{i=1}^n x_i^{a_i(\lambda)}$ where $z_\lambda=\prod_{i=1}^ni^{a_i(\lambda)}a_i(\lambda)!$\label{25}

\item 
For $\Gamma\subseteq S_k,\ H\subseteq S_n,\ G=\Gamma^n\rtimes H\subseteq S_{kn}$ one has $Z_G(x_{1},\dots,x_{kn})=Z_H(t_1,\dots,t_n)$\label{26} where $t_i=Z_\Gamma(x_{1i},\dots,x_{ki})$

\item $\sum_{n=0}^\infty Z_{S_n}(x_1,\dots,x_n)t^n=\exp\left\{\sum_{i=1}^\infty \frac{t^i}{i}x_i\right\}$ \label{27}

\end{enumerate}

P\'olya studied the problem of classifying molecules up to symmetry. P\'olya theory has been used for graphical enumeration, as in the number of unlabeled trees. A striking example is Hanlon's \cite{hanlon81,hanlon84} study of graph coloring. Another is Craven's work on enumeration of quadratic forms \cite{craven1980application}.

We have not seen many applications in probability theory. Yet, the cycle index has a perfectly simple probabilistic interpretation: for $G\subseteq S_n$, suppose $s\in G$ has cycle type $a_i(s)$, $1\leq i\leq n$. Then
\begin{equation*}
P_G(a_i(s) = a_i; 1\leq i\leq n)=\left[\prod x_i^{a_i}\right]C_G(x_1,\dots,x_n).
\end{equation*}

\begin{exam}
Consider $G=C_2^2\rtimes S_2\subseteq S_4$, so $|G|=8$. The elements of $G$ are
\begin{equation*}\begin{array}{lcccccccc}
\text{cycle type}&1^4&1^22&1^22&2^2&2^2&4&4&2^2\\\hline
&1234&1234&1234&1234&1234&1234&1234&1234\\
&1234&2134&1243&2143&3412&4312&3421&4321
\end{array}
\end{equation*}
Hence, from the definition,
\begin{equation*}
Z_G(x_1,x_2,x_3,x_4)=\frac18\{x_1^4+2x_1^2x_2+3x_2^2+2x_4\}.
\end{equation*}
Using $Z_{C_2}(x_1^2,x_2)=(x_1+x_2)/2=Z_{S_2}(x_1,x_2)$, formula \eqref{26} gives
\begin{equation*}\begin{gathered}
Z_G(x_1,x_2,x_3,x_4)=\left(\frac{t_1^2+t_2}2\right),\\
t_1=\frac{x_1^2+x_2}2,\quad t_2=\frac{x_2^2+x_4}2,
\end{gathered}\end{equation*}
so
\begin{equation*}
Z_G(x_1,x_2,x_3,x_4)=\frac12\left[\left(\frac{x_1^2+x_2}2\right)^2+\frac{x_2^2+x_4}2\right]
=\frac18\{x_1^4+2x_1^2x_2+3x_2^2+2x_4\}
\end{equation*}
as above. In particular: pick $s\in G$ uniformly, the chance of a four-cycle is $2/8=1/4$.
\end{exam}

\begin{exam}
Consider $G=S_k^n\rtimes S_n$. Claim:
\begin{equation*}
P_G(kn\text{-cycle})=\frac1{kn},\qquad P_G(1^{kn})=\frac1{(k!)^nn!}
\end{equation*}
To see the first claim, write
\begin{equation*}
Z_G(x_1,\dots,x_{kn})=\sum_{\lambda \vdash n}\frac1{z_\lambda}\prod_{i=1}^n Z_{S_k}(x_i,x_{2i},\dots,x_{ki})^{a_i(\lambda)}.
\end{equation*}
The only term on the right side that carries the monomial $x_{kn}$ is the term corresponding to $\lambda=(n)$, $i=n$, with coefficient $(1/n)(1/k)=1/(nk)$ as claimed. The proof for $1^{kn}$ is easier.
\end{exam}

\begin{exam}
Similarly, consider $G=C_k^n\rtimes S_n$.
\begin{equation*}
P_G(kn\text{-cycle})=\frac1{k}\frac{\phi(k)}{n},\qquad P_G(1^{kn})=\frac1{k^nn!}
\end{equation*}
\end{exam}

\begin{exam}
Consider $G=C_k^n\rtimes C_n$, with $n$ and $k$ relatively prime.
\begin{equation*}
P_G(kn\text{-cycle})=\frac{\phi(k)}{k}\frac{\phi(n)}{n}=\frac{\phi(nk)}{nk},\qquad P_G(1^{kn})=\frac1{k^nn}
\end{equation*}
\end{exam}

\noindent The next subsection gives a more substantive example.

\subsection{Distribution of the number of cycles in $\Gamma^n\rtimes H$}\label{sec25}

The number of cycles $C(\sigma)$ of a random permutation in $S_n$ is a classical object of study \cite{gonch42,gonch44}; see also \cite{feller}. These papers prove a central limit theorem:

\begin{thm}[Goncharov]
Let $\sigma$ be chosen from the uniform distribution on $S_n$. Let $C(\sigma)$ be the number of cycles. Then, with $\gamma=$ Euler's constant,
\begin{align*}
E_n(C(\sigma))&=\log n+\gamma+O\left(\frac1{n}\right),\\
\var_n(C(\sigma))&=\log n+\gamma-\frac{\pi^2}6+O\left(\frac1{n}\right),
\end{align*}
and, normalized by its mean and variance, $C(\sigma)$ has a limiting standard normal distribution.
\end{thm}

This subsection studies the distribution of the number of cycles in Wreath products. Let $\Gamma$ be a subgroup of $S_k$ and $H$ be a subgroup of $S_n$. Then $\Gamma^n\rtimes H$ acts on $1,\dots,kn$ by permuting independently within blocks of size $k$ by the elements in $\Gamma^n$ and then permuting the blocks by $h\in H$.

\paragraph*{The cycle index connection}\quad Recall the cycle index:
\begin{equation*}
Z_\Gamma(x_1,\dots,x_k)=\frac1{|\Gamma|}\sum_{\gamma\in\Gamma}\prod_{i=1}^kx_i^{a_i(\gamma)}.
\end{equation*}
Setting all $x_i=x$ gives the generating function for the number of cycles, $\sum_{i=1}^k a_i(\gamma)=C(\gamma)$; call this generating function $C_\Gamma(x)$,
\begin{equation*}
C_\Gamma(x)=\sum_{j=1}^kP_\Gamma(j)x^j,\qquad P_\Gamma(j)=P(C(\gamma)=j).
\end{equation*}
From \eqref{26},
\begin{equation*}
C_{\Gamma^n\rtimes H}(x)=\sum_{j=1}^nP_H(j)C_\Gamma(x)^j.
\end{equation*}
This proves:

\begin{prop}
For $\sigma$ uniformly chosen in $\Gamma^n\rtimes H$, the number of cycles $C(\sigma)$ has the same distribution as $\sum_{i=1}^NX_i$, where $X_i$, $1\leq i<\infty$, are i.i.d.\ with generating function $C_\Gamma(x)$ and $N$ is independent of $\{X_i\}$ with generating function $C_H(x)$.
\end{prop}

This represents $C(\sigma)$ as a randomly stopped sum, an extensively studied class of random variables; see \cite{gut2009,gut2012}. In particular,
\begin{align*}
E_{\Gamma^n\rtimes H}(C)&=E_\Gamma(C)E_H(C),\\
\var_{\Gamma^n\rtimes H}(C)&=E\left(\var_H(C\mid N)\right)+\var_H\left(E(C\mid N)\right),\\
&=E(H)\var_H(C)+(E_H(C))^2\var(H).
\end{align*}

As shown below, a classical central limit theorem of Anscomb for stopped sums will give a central limit theorem under conditions on $H$ when $n$ is large. See \cite{gut2009,gut2012} for a full discussion of Anscomb's theorem.

For $\Gamma$ or $H$ a full symmetric group, the means and variances are given in the preceding theorem from Goncharov. We pause to compute them for the important case where $\Gamma$ is the cyclic group $C_k$.

\begin{prop}
Let $\mu_k,\nu_k$ be the first and second moments of the number of cycles of a randomly chosen element of the cyclic group $C_k$. Then $\mu_k,\nu_k$ are multiplicative functions of $k$; recall $f$ is multiplicative if $f(ab)=f(a)f(b)$ for GCD $(a,b)=1$. Further, for $p$ a prime, $a\geq1$,
\begin{equation*}
\mu_{p^a}=1+a\left(1-\frac1{p}\right),\qquad\nu_{p^a}=p^a\left[1+\frac1{p}\left(1-\frac1{p^{a-1}}\right)\right].
\end{equation*}
This determines $\mu_k,\nu_k$ for all $k$.
\end{prop}

\begin{proof}
A randomly chosen $\sigma$ in $C_n$ has $C(\sigma)=n/d$ with probability $\phi(d)/n$ if $d\mid n$. It follows that
\begin{align*}
E_n(C)&=\sum_{d\mid n}\frac{n}{d}\frac{\phi(d)}{n}=\sum_{d\mid n}\frac{\phi(d)}{d},\\
E_n(C^2)&=\sum_{d\mid n}\left(\frac{n}{d}\right)^2\frac{\phi(d)}{n}=n\sum_{d\mid n}\frac{\phi(d)}{d^2}.
\end{align*}
Elementary number theory shows $\phi(d)$ and $d$ are multiplicative and, if $f$ is multiplicative, $\sum_{d\mid n}f(d)$ is, too. This proves the first claim.

For $n=p^a$, the divisors are $p^j$, $0\leq j\leq a$, and $\phi(1)=1$, $\phi(p^j)=p^j-p^{j-1}$ for $j\geq1$. Thus
\begin{equation*}
E_{p^a}(C)=1+\sum_{j=1}^a\frac{p^j-p^{j-1}}{p^j}=1+a\left(1-\frac1{p}\right)
\end{equation*}
as claimed. The claim for $E_{p^a}(C^2)$ is proved similarly.
\end{proof}

\noindent In current applications $\Gamma\subseteq S_k$ with $k$ fixed, so the asymptotics of $E_\Gamma(C)$, $\var_\Gamma(C)$ are not a problem.

Clearly some conditions on $H\subseteq S_n$ are required to have a central limit theorem; if $H=C_n$ and $\Gamma=S_1$, $\Gamma^n\rtimes C_n$ acts on $\{1,\dots,n\}$ in the usual way by cycling $\pmod n$. The number of cycles is $d$ with probability $\phi(n/d)/n$. This is a discrete distribution which depends heavily on the divisibility properties of $n$. No simple limit theorem seems possible.

We content ourselves with the following, taking general $\Gamma\subseteq S_k$ for $k$ fixed and $H=S_n$. This includes all of our introductory examples and the argument generalizes to groups such as $H=B_n$.

\begin{thm}
Let $\Gamma\subseteq S_k$ be fixed and consider $\Gamma^n\rtimes S_n$. For $\eta$ chosen uniformly in $\Gamma^n\rtimes S_n$, the number of cycles of $\eta$ acting on $\{1,\dots,kn\}$ satisfies
\begin{equation*}
\frac{C(\eta)-\mu_n}{\sigma\sqrt{\log n}}\Longrightarrow\caln(0,1),
\end{equation*}
where $\mu_n=\mu(\log n+\gamma+O(n^{-1}))$ and $\mu,\sigma^2$ are the mean and variance of the number of cycles of a random element of $\Gamma$.
\end{thm}

\begin{proof}
Let $X_1,X_2,\dots$ be i.i.d.\ random variables chosen from the distribution with generating function $C_\Gamma(x)$ with means subtracted. Let $N_n$ be chosen from $C_{S_n}(x)$. From Goncharov's theorem,
\begin{equation*}
\frac{N_n}{\log n}\Pto 1\qquad\text{as }n\to\infty.
\end{equation*}
Let $S_m=X_1+\dots+X_m$. The classical central limit theorem implies, as $m\to\infty$,
\begin{equation*}
\frac{S_m}{\sqrt{m}}\Longrightarrow\caln(0,\sigma^2).
\end{equation*}
Write $n_0=\lfloor\log n\rfloor$. Then
\begin{equation*}
\frac{S_{N_n}}{\sqrt{N_n}}=\left(\frac{S_{n_0}}{\sqrt{n_0}}+\frac{S_{N_n}-S_{n_0}}{\sqrt{n_0}}\right)\left(\frac{n_0}{N_n}\right)^{1/2}.
\end{equation*}
It must only be shown that
\begin{equation*}
\frac{S_{N_n}-S_{n_0}}{\sqrt{n_0}}\Pto 0\qquad\text{as }n\to\infty.
\end{equation*}
For this, let $\epsilon\in(0,1/3)$, set $n_1=\lfloor n_0(1-\epsilon^3)\rfloor+1$ and $n_2=\lfloor n_0(1+\epsilon^3)\rfloor$. Write
\begin{align*}
P\left\{|S_{N_n}-S_{n_0}|>\epsilon n_0^{1/2}\right\}&=P\left\{|S_{N_n}-S_{n_0}|>\epsilon n_0^{1/2}\cap\{N_n\in[n_1,n_2]\}\right\}\\
&\qquad+P\left\{|S_{N_n}-S_{n_0}|>\epsilon n_0^{1/2}\cap\{N_n\notin[n_1,n_2]\}\right\}\\
&\leq P\left\{\max_{n_1\leq k\leq n_0}|S_k-S_{n_0}|>\epsilon n_0^{1/2}\right\}\\
&\qquad+P\left\{\max_{n_0\leq k\leq n_2}|S_k-S_{n_0}|>\epsilon n_0^{1/2}\right\}+P\{N_n\notin[n_1,n_2]\}.
\end{align*}
Using Kolmogorov's inequality (see Wikipedia entry) for the first two terms, this is bounded above by
\begin{equation*}
\frac{(n_0-n_1)\sigma^2}{n_0}+\frac{(n_2-n_0)\sigma^2}{n_0}+P\{N_n\notin[n_1,n_2]\}
<2\epsilon\sigma^2+P\{N_n\notin[n_1,n_2]\}.
\end{equation*}
From $N_n/n_0\pto1$, the last term tends to zero and the proof is complete.
\end{proof}

\begin{note}
The only property of the cycle distribution of $H$ that was used in the proof was that $E_H(C)\to\infty$ and that $C$ is concentrated about its mean. Thus the theorem holds for, e.g., $H=B_n\subseteq S_{2n}$ and many other choices.
\end{note}

\section{A General Theorem}\label{sec3}

This section derives a theorem for the joint distribution of the number of $i$ cycles $a_i(\sigma)$ for $\sigma\in\Gamma^n\rtimes S_n$. The theorem is given in two forms:
\begin{enumerate}
\item an exact result for $\{a_i\}_{i=1}^\infty$ when $n$ is randomized;
\item a limiting result when $n$ is large.
\end{enumerate}

\begin{thm}
Fix $k$ and $\Gamma\subseteq S_k$. Let $G_n=\Gamma^n\rtimes S_n$. Let $G_\infty$ be the union of $G_n$, $0\leq n<\infty$. For $0<t<1$, define $U_t$ on $G_\infty$ as follows: pick $N\in\{0,1,2,\dots\}$ with $P(N=n)=(1-t)t^n$, then pick $\sigma\in G_n$ uniformly. For $\sigma\sim\prod_{1\leq i<\infty}i^{a_i(\sigma)}$, under $U_t$, the joint distribution of $\{a_i(\sigma)\}_{i=1}^\infty$ is the same as the joint distribution of
\begin{align}
A_i&=\sum_{\substack{\lambda\vdash k\\ jl=i}}a_j(\lambda)Z_{l, \lambda}\label{31}\\
Z_{l, \lambda}&\sim\poi\left(\frac{t^l}{l}P_\Gamma(\lambda)\right),\qquad\text{with $Z_{l, \lambda}$ independent.}\label{32}
\end{align}
\label{thm31}
\end{thm}

The sum in \eqref{31} is over partitions of $k$ and $j,l \in \mathbb{N}$ that multiply to $i$, and so is finite. $P_\Gamma(\lambda)$ is the probability that a uniformly chosen element of $\Gamma$ has cycle type $\lambda$.

Note, as in Example 1 of the Introduction, the same $Z_{l, \lambda}$ may appear in several of the $A_i$ so these are dependent, compound Poisson variates as in \ref{sec23}.

\begin{thm}
Fix $k$ and $\Gamma\subseteq S_k$. Let $G_n=\Gamma^n\rtimes S_n$. Pick $\sigma\in G_n$ uniformly and let $\sigma\sim\prod_{i=1}^{kn}i^{a_i(\sigma)}$. Then, as $n\nearrow\infty$, the joint distribution of $\{a_i(\sigma)\}_{i=1}^{kn}$ converges (weakly) to the law of $\{A_i\}_{i=1}^\infty$ with $A_i$ as in \eqref{31},\eqref{32} with $t=1$.
\label{thm32}
\end{thm}

A quantitative form of \ref{thm32}, showing that the total variation distance between the laws $\{a_i\}_{i=1}^b$, $\{A_i\}_{i=1}^b$ is at most $2b/n$ is in \ref{sec4}.

The proofs of \ref{thm31} and \ref{thm32} use moments. It may help the reader to recall that if $X,Y,Z$ are independent Poisson $\lambda,\mu,\nu$, then
\begin{align*}
\bullet\ &E(x^{jX})=\exp\{\lambda(x^j-1)\}\\
\bullet\ &\text{For }A=jX+lZ,\ B=kY+l'Z,\ j,k,l,l'\in\bbn_+,\\
         &E(x^Ay^B)=\exp\{\lambda(x^j-1)+\mu(y^k-1)+\nu(x^ly^{l'}-1)\}
\end{align*}

Similarly, if for every non-empty subset $s\subseteq[n]$, $X_s$ are independent Poisson $(\lambda_s)$ and $C_s^i\in\bbn$, $W_i=\sum_sC_s^iX_s$, then
\begin{equation*}
E\left(\prod_{i \in [n]} x_i^{W_i}\right)=\exp\left\{\sum_{s\subseteq[n]}\lambda_s\left(\prod_{i\in s}x_i^{C_s^i}-1\right)\right\}.
\end{equation*}

\begin{exam}
The notation in \ref{thm31} and \ref{thm32} is daunting. Consider \ref{thm32} with $\Gamma=S_3$ as in \ref{thm11} of the Introduction. Then, $\lambda=1^3,21,3$ with respective chances
\begin{equation*}
P_{S_3}(1^3)=1/6,\quad P_{S_3}(21)=1/2,\quad P_{S_3}(3)=1/3.
\end{equation*}

Consider $A_1$. The sum in \eqref{31} has only the term $j=l=1$, so $A_1=3Z_{1,1^3}+Z_{1,21}$ with $Z_{1,1^3}\sim\poi(1/6)$, $Z_{1,21}\sim\poi(1/2)$.

Consider $A_2$. The sum has only the terms $j=1$, $l=2$ or $j=2$, $l=1$. So $A_2=Z_{1,21}+3Z_{2,1^3}+Z_{2,21}$. The three terms are independent with
\begin{equation*}
Z_{1,21}\sim\poi(1/2),\quad Z_{2,1^3}\sim\poi(1/12),\quad Z_{2,21}\sim\poi(1/4).
\end{equation*}
This matches the claims in \ref{thm11}.
\end{exam}

\begin{proof}[Proof of \ref{thm31}]
From the cycle index facts of \ref{sec23} and \ref{sec24},
\begin{equation*}
\sum_{n=0}^\infty Z_{G_n}(x_1,x_2,\dots,x_{kn})(1-t)t^n=\exp\left\{\sum_{a=1}^\infty\frac{t^a}{a}\left(Z_\Gamma(x_a,x_{2a},\dots,x_{ka})-1\right)\right\}.
\end{equation*}
Expanding the exponent on the right side gives
\begin{equation*}
\sum_{a=1}^\infty \frac{t^a}{a}\sum_{\lambda\vdash k}P_\Gamma(\lambda)\left(\prod_{b=1}^kx_{ab}^{a_b(\lambda)}-1\right)=\sum_{\lambda\vdash k}\sum_{a=1}^\infty\frac{t^a}{a}P_\Gamma(\lambda)\left(\prod_{b=1}^kx_{ab}^{a_b(\lambda)}-1\right).
\end{equation*}
The right side is the log of the generating function of the dependent compound Poisson variates $A_i$ in \eqref{32} picking off all variates $x_{ab}$ where $ab=i$.
\end{proof}

\begin{proof}[Proof of \ref{thm32}]
As $t\to1$, the right-hand side of \eqref{32} converges to the generating function of the claimed compound Poisson distributions. It must be shown that the coefficient of $t^n$ on the left-hand side converges to this same limit. This may be argued from the Tauberian arguments of \cite{shepp} or \cite{fris}.

However, \ref{sec4} below gives an independent coupling argument showing that, for any fixed $m$, the joint distribution of $a_1(\sigma),\dots,a_m(\sigma)$ converges to some limit law as $n$ tends to infinity. Thus, setting $x_{m+1},x_{m+2},\dots$ equal to 1 and letting $E_n=E_{S_n}(\prod_{i=1}^mx_i^{a_i(\sigma)})$, the generating functions $E_n$ converge to a limit $E$ as $n$ tends to infinity. Now, Able's theorem shows that
\begin{equation*}
\lim_{t\to1}\sum_{n=0}^\infty E_n(1-t)t^n\longrightarrow E.
\end{equation*}
By inspection, from \ref{thm31}, $E=E(\prod_{i=1}^mx_i^{A_i})$. This completes the proof.
\end{proof}

\section{Coupling for Wreath Products}\label{sec4}

\subsection{Introduction}\label{sec41}

This section gives a coupling proof of the convergence of the cycle lengths of a random permutation $\sigma\in\Gamma^n\rtimes S_n$ to the compound Poisson distributions of \ref{thm32}. The proof comes with a rate of convergence which seems difficult to derive from the generating function approach of \ref{sec3}. The argument here is also needed to complete the proof of \ref{thm32}.

Here is a special case of the main result. Let $\Gamma\subseteq S_k$. Suppose $\sigma$ has $a_i(\sigma)$ $i$-cycles as a permutation of $\{1,\dots,kn\}$. Let $\{A_i\}_{i=1}^\infty$ be the compound Poisson distributions of \ref{thm32}.

\begin{thm}
With notation as above, for $\sigma$ uniform in $\Gamma^n\rtimes S_n$, for any $b,k,n\in\bbn$, let $\mu$ be the distribution of $\{a_i(\sigma)\}_{i=1}^b$ and $\nu$ be the distribution of $\{A_i\}_{i=1}^b$. Then
\begin{equation*}
\|\mu-\nu\|_{\tv}\leq\frac{2b}{n}.
\end{equation*}
\label{thm41}
\end{thm}

\ref{thm41}, and variations when $k$ grows and $n$ is fixed, are proved by coupling. The coupling is reminiscent of the well known Feller coupling which we now recall.

Build a uniformly distributed permutation $\sigma\in S_n$ directly in cycle form as follows. 
Starting with $1$, choose $\sigma(1)$ uniformly from $[n]$. Then choose $\sigma(\sigma(1))$ uniformly from $[n] \setminus \sigma(1)$. Continue in this manner until $\sigma(i) = 1$, when the cycle is completed, and then start over with all the unused elements. With $n=6$, a typical construction is
$$
(1 \quad (14 \quad (146 \quad (146)(2 \quad (146)(2)(3 \quad (146)(2)(35 \quad (146)(2)(35)
$$
with steps 0 through 6. At step three, $1$ is chosen as the image of $6$, so the cycle closes and another begins with an arbitrary element out of $\set{2,3,5}$. In step four this element is chosen as the image of itself, creating a fixed point and continuing with a new cycle. If $\zeta_i$ is the indicator that a new cycle is started at step $n-i+1$, clearly $P(\zeta_i=1)=1/i$, $1\leq i\leq n$. 

The joint distribution of the cycle lengths can be neatly described just using the independent random variables $\zeta_i$. Following the careful description of \cite{najnudel2020feller}, say an $l$-spacing occurs in a binary sequence $b_1,b_2,\dots$ starting at position $i-l$ and ending at position $i$ if $b_{i-l},\dots,b_i=10^{l-1}1$. If $C_l$ is the number of $l$ spacings in $\zeta_1,\dots,\zeta_n100\dots$, then
\begin{equation*}
\{C_l\}_{l=1}^n\overset{L}{=}\{a_l(\sigma)\}_{l=1}^n,
\end{equation*}
where $\sigma$ is uniformly chosen in $S_n$. A coupling argument shows that the limit distribution of cycle counts can then be read from an infinite random binary string $\zeta=\zeta_1\zeta_2\dots\zeta_n\zeta_{n+1}\dots$ with $P(\zeta_i = 1) = 1/i$. A surprising theorem of Ignatov (see \cite{najnudel2020feller}) says that the number of $l$ blocks in $\zeta$ is (exactly) $\poi(1/l)$ and that these are independent as $l$ varies (!). All of this and extensions to Ewens distributed random permutations is wonderfully explained in \cite{najnudel2020feller} which also includes a scholarly review of the literature.

In \ref{sec42} we give a similar representation of the distribution of $\{a_i(\sigma)\}$ using binary indicators for $\sigma\in\Gamma^n\rtimes S_n$. The proof of \ref{thm41} and several extensions is in \ref{sec43}. The full coupling is described in \ref{sec44} with examples.

\subsection{Cycles from binary indicators}\label{sec42}


For clarity we single out here the simplified construction to sample the cycle type of a random $\sigma\in\Gamma^n\rtimes S_n$ using binary indicators. Showing that these indicators indeed arise from sampling the permutation itself is relegated to \ref{sec44}. Let $\zeta_1,\dots,\zeta_n$ be distributed $\zeta_i \sim \text{Ber}(1/i)$ independently and $Y_1,\dots,Y_n$ be the cycle types of iid $\Gamma$-permutations. Let $E_{l,i}$ be the event that there is an $l$-space in $\zeta_1 \dots \zeta_n1$ starting at position $i$.


\begin{prop}
With the construction above, for $1\leq b\leq kn$, set
\begin{equation}
C_b \coloneqq \sum_{\substack{jl = b}} \sum_{i=1}^n a_j(Y_i)1_{E_{l,i}}
\label{41}
\end{equation}
Then there is equality of joint distributions,
\begin{equation*}
\{C_b\}_{b=1}^{kn}\overset{L}{=}\{a_b(\sigma\}_{b=1}^{kn},
\end{equation*}
where $\sigma$ is chosen uniformly in $\Gamma^n\rtimes S_n$ and has $a_i(\sigma)$ $i$-cycles.
\label{prop41}
\end{prop}

\begin{note}
The $\{C_b\}$ (and so the cycle indices $\{a_b(\sigma)\}$) have somewhat complex dependence, even in the large-$n$ limit (see \ref{thm11}).
\end{note}


\ref{prop41} is proved in \ref{sec44}. It is used in \ref{sec43} to give a coupling argument showing that $\{a_i(\sigma)\}_{i=1}^b$ has a limit. This will complete the proof of \ref{thm32}.

\subsection{Coupling to the limit}\label{sec43}

This section provides a coupling proof of \ref{thm41} and some refinements. Throughout, let $\zeta_i$ be independent Bernoulli with $P\{\zeta_i=1\}=1/i$, $1\leq i<\infty$. 
\ref{prop41} of \ref{sec42} uses these binary indicators (and some independent auxiliary variates $Y_i$) to construct random variables $C_b$ shown to be equidistributed with the number of $b$-cycles in a uniform element of $\Gamma^n\rtimes S_n$. The construction only needs $\zeta_1,\dots,\zeta_n$, but the formula defining $C_b$ works for infinite sequences as well. Let $C_b^\infty$ be the random variables constructed from the infinite sequence by the formulae of \ref{prop41}. 


\begin{thm}
For any $B,k,n\in\bbn$, $\Gamma\subseteq S_k$, let $\mu$ be the distribution of $C_b$, $1\leq b\leq B$, and $\nu$ be the distribution of $C_b^\infty$, $1\leq b\leq B$. Then
\begin{equation*}
\|\mu-\nu\|_{\tv}\leq\frac{2B}{n}.
\end{equation*}
\label{thm42}
\end{thm}

\begin{proof}
Mirroring the proof of \cite[Lemma 1.4]{arratia03} note that $C_b>C_b^\infty$ is only possible due to the deterministic $\zeta_{n+1}=1$ for $C_b$, while for $C_b^\infty$ the $(n+1)$th indicator term remains random. This can only lead to extra cycles with lengths in $\{1,\dots,B\}$ if the space preceding this 1 is of length at most $B$. For fixed $i\leq n$,
\begin{equation*}
P\{\zeta_{n-i+1} \dots \zeta_n\zeta_{n+1}=10^{i-1}1\}=\frac1{n-i+1}\frac{n-i+1}{n-i+2}\dots\frac{n}{n+1}=\frac1{n+1}.
\end{equation*}
Now, a union bound shows the probability that the number of zeros preceding $\zeta_{n+1}=1$ is of length at most $B$ is at most $B/(n+1)$.

Similarly, for $b\in\{1,\dots,B\}$, $C_b < C_b^\infty$ can only occur if a $b$ spacing in $\zeta$ occurs after $n$. Now,
\begin{equation*}
\sum_{k>n+1}P\{\zeta_{k-b}\dots\zeta_k=10^{b-1}1\}=\sum_{k>n+1}\frac1{(k-1)k}=\frac1{n+1}.
\end{equation*}
Thus, union bounding over $1\leq b\leq B$ gives that the probability of an extra space in the tail is at most $B/(n+1)$. Combining bounds,
\begin{equation*}
P\left\{C_b\neq C_b^\infty\text{ for any }1\leq b\leq B\right\}\leq\frac{2B}{n+1}.\qedhere
\end{equation*}
\end{proof}

\begin{cor}[Proof of \ref{thm41}]
The preceding argument shows that the joint distribution of $\{C_b\}_{b=1}^{kn}$ converges to the joint distribution of $\{C_b^\infty\}_{b=1}^\infty$ (in the sense of finite dimensional distributions). \ref{prop41} identifies the first with the joint distribution of $\{a_i(\sigma)\}_{i=1}^{kn}$ with $\sigma$ uniform in $\Gamma^n \rtimes S_n$. \ref{thm32} shows that the Abel limit of the generating function of $\{a_i(\sigma)\}_{i=1}^{kn}$ converges to the generating function of the compound Poisson distributions of \ref{thm41}. Since \ref{thm42} shows that the sequence of random variables $\{a_i(\sigma)\}_{i=1}^{kn}$ converge to ``something'' (here $\{C_b^\infty\}_{b=1}^\infty$) that something must be the claimed compound Poisson limits.\hfill{$\square$}
\label{cor41}
\end{cor}

\ref{thm42} is uniform in $k$ but requires $n \to \infty$. For the special case $\Gamma=S_k$, it is possible to obtain a convergence theorem for fixed $n$ as $k\to\infty$. In this case we sample the $Y_i$ of \ref{sec42} using sequences $\eta^i = \eta^i_1,\ldots,\eta^i_k$ just as we did with $\zeta$ to sample from $S_n$, that is $\eta^i_j$ are Bernoulli with $P(\eta^i_j = 1) = 1/j$ independently. The fact that the number of $l$-spaces in $\eta^i$ is equal in distribution to the number of $l$-cycles in $S_k$ follows from \ref{prop41} with $\Gamma = S_1$ (and is also just the original Feller coupling). Then we define $C_b^{\infty,n}$ in analogy with $C_b^\infty$: letting $F^i_{l,j}$ be the event that there is an $l$-space in $\eta^i$ starting at $j$
\begin{equation*}\label{eq:cbkn}
C_b^{k,n} \coloneqq \sum_{\substack{jl = b}} \sum_{i=1}^n \sum_{m=1}^k 1_{E_{j,i}}1_{F^i_{l,m}}
\end{equation*}

\begin{thm}
Let $\sigma$ be uniform in $S_k^n\rtimes S_n$ of cycle type $\{a_i(\sigma)\}_{i=1}^{kn}$. For fixed $B,k,n$, let $\mu$ be the joint distribution of $\{a_b(\sigma)\}_{b=1}^B$. Let $\nu$ be the joint distribution of $\{C_b^{\infty,n}\}_{b=1}^B$. Suppose finally that $k\geq89>e^{e^{3/2}}$. Then
\begin{equation*}
\|\mu-\nu\|_{\tv}\leq\frac{5B\log B\log k}{k}.
\end{equation*}
\label{thm43}
\end{thm}

\ref{thm43} is uniform in $n$ and so we may allow $k=f(n)$ with $n\to\infty$. Thus for $\sigma$ uniform in $S_{f(n)}^n\rtimes S_n$, $\mu$ as above, and $\nu$ the joint distribution of $\{C_b^{\infty,\infty}\}_{b=1}^B$:
\begin{cor}
\begin{equation*}
\|\mu-\nu\|_{\tv}\leq\frac{5B\log B\log f(n)}{f(n)}+\frac{2B}{n}.
\end{equation*}
\label{cor42}
\end{cor}

\begin{proof}[Proof of \ref{thm43}]
Let $\zeta=\zeta_1,\zeta_2,\dots$. By Ignatov's theorem, the number of $i$-spacings in $\zeta$ are independent $\poi(1/i)$, $1\leq i<\infty$. From the coupling, the number of spacings in $1\zeta_1\dots\zeta_n1$ of size at most $B$ is stochastically dominated by $1+\sum_{i=1}^BX_i$, where $X_i$ are independent $\poi(1/i)$. This in turn is stochastically dominated by $X\sim\poi(2\log B) + 1$. By Bennett's inequality,
\begin{align*}
P\{X>2\log B\log k + 1\}& \le \exp\left\{-2\log B\left(\log k\log\log k-\log k\right)\right\}\\
& \le \exp\{-\log B\log k\},
\end{align*}
where $k\geq e^{e^{3/2}}$. Let $E$ be the event that there are at most $2\log B\log k$ spacings in $\zeta_1\zeta_2\dots\zeta_n1$ of size at most $B$. Then \ref{thm42} may be applied for each $\eta^i$ with a union bound to give
\begin{equation*}
P\{C_b^{k,n}\neq C_b^{\infty,n}\text{ for any }1 \leq b\leq B | E\} \leq\frac{4B\log B\log k}{k}.
\end{equation*}
This yields
\begin{align*}
P\{C_b^{k,n}\neq C_b^{\infty,n}\text{ for any }1\leq b\leq B\}&\leq\frac{4B\log B\log k}{k}+\exp\{-\log B\log k\}\\
&\leq\frac{5B\log B\log k}{k}.\qedhere
\end{align*}
\end{proof}

We further have an explicit description of the limit law of \ref{thm43}, which is proved in \cite{tung2025cutting}.

\begin{thm}
    Let $\sigma$ be uniform in $S_k^n\rtimes S_n$ of cycle type $\{a_i(\sigma)\}_{i=1}^{kn}$. As both $k,n \nearrow \infty$ the joint distribution of $\{a_i(\sigma)\}_{i=1}^{kn}$ converges (weakly) to the law of $\{A_i\}_{i=1}^\infty$ with 
    $$
     A_i = \sum_{kl = i} \sum_{j=1}^\infty j X_{l,k,j}
    $$
    where $X_{l,k,j} \sim \text{Poiss}\left(\frac{p^k_j}{l}\right)$ with $p^{k}_j$ the Poisson PMF with parameter $1/k$ evaluated at $j$. $\set{X_{l,k,j}}_j \cup \set{X_{l',k',j}}_j$ are mutually independent if $l \neq l'$, but otherwise may be dependent.
    \label{thm:lim}
\end{thm}

As appealing and useful as these coupling bounds are, it must be remembered that they can be exponentially far off. For $\sigma$ uniform in $S_n$, $\mu$ the law of $a_1(\sigma)$, $\nu=\poi(1)$,
\begin{equation*}
\|\mu-\nu\|_{\tv}\leq\frac{2^n}{(n+1)!}.
\end{equation*}
See \cite{pd2023} for detailed discussion.

\subsection{A Feller coupling for Wreath products}\label{sec44}

This section provides a direct sequential description of a uniform $\sigma\in\Gamma^n\rtimes S_n$ in cycle form. Keeping track of the steps of the algorithm yields the binary indicators $\zeta_i$ (and $\eta^i_j$) used previously.

Recall that $\Gamma\subseteq S_k$, $\sigma=(\gamma_1,\dots,\gamma_n;\eta)$ for $\gamma_i\in\Gamma$, $\eta\in S_n$. Then $\sigma$ acts on $\{1,\dots,kn\}$ by using $\gamma_1$ to permute $\{1,\dots,k\}$, $\gamma_2$ to permute $\{k+1,\dots,2k\}$, $\dots,\gamma_n$ to permute $\{(n-1)k+1,\dots,nk\}$, and finally $\eta$ permutes these $n$ blocks.

\begin{exam}
When $k=3,n=4$, $\left[(132),(1)(23),(312),(1)(2)(3);(1432)\right]$ acts on $\{1,\dots,12\}$ by first permuting within blocks as
\begin{equation*}\begin{array}{cccccccccccc}
1&2&3&4&5&6&7&8&9&10&11&12\\
3&1&2&4&6&5&8&9&7&10&11&12
\end{array}
\end{equation*}
and then using $\eta$ on the blocks to get
\begin{equation*}
\sigma=\begin{array}{cccccccccccc}
1&2&3&4&5&6&7&8&9&10&11&12\\
10&11&12&3&1&2&4&6&5&8&9&7
\end{array}.
\end{equation*}
In cycle notation $\sigma=(1\,10\,8\,6\,2\,11\,9\,5)(3\,12\,7\,4)$ with $a_4(\sigma)=1$, $a_8(\sigma)=1$.
\end{exam}

The coupling depends on a sequential procedure for building up the cycle representation directly, one step at a time. This is described below, where the case $S_3^4\rtimes S_4$ will be used as a running example.

\begin{itemize}
\item Begin by putting the $n$ blocks $\{1,\dots,k\},\{k+1,\dots,2k\}\dots$ in an urn. The $n$ fixed places for these blocks will be called ``block places''.

\item Pick a block at random from the urn, remove it, permute the elements by a uniformly chosen $\gamma\in\Gamma$, and place it under place 1. In the example, if block 4 is chosen first, and $\gamma=$ id, the permutation starts
\begin{equation*}\begin{array}{cccccccccccc}
1&2&3&4&5&6&7&8&9&10&11&12\\
10&11&12
\end{array}.
\end{equation*}

\item Construct a bookkeeping array of ``open cycles'' $(1,\gamma_1 ; (2,\gamma_2 ; \dots ; (k,\gamma_k$. In the example: $(1,10 ; (2,11 ; (3,12$.

\item Pick a random remaining block from the urn, permute it by a uniformly chosen $\gamma\in\Gamma$, and place it under the previously chosen block. Then extend the $k$ open cycles in the bookkeeping array. In the example, this results in
\begin{equation*}\begin{array}{cccccccccccc}
1&2&3&4&5&6&7&8&9&10&11&12\\
10&11&12& & & & & & &8&9&7
\end{array}
\end{equation*}
if the block containing $7 8 9$ is chosen second and $\gamma$ permutes this to $8 9 7$. The bookkeeping array becomes $(1,10,8 ; (2,11,9 ; (3,12,7$.

\item Continue in this fashion, extending the open cycles in the bookkeeping array until the block in place 1 is pulled from the urn, forcing them to close. In the example this may look like
\begin{equation*}\begin{array}{cccccccccccc}
1&2&3&4&5&6&7&8&9&10&11&12\\
10&11&12& & & &2&1&3&8&9&7
\end{array}
\end{equation*}
and $(1,10,8);(2,11,9,3,12,7)$.


\item If there are more blocks left in the urn, start again by setting the leftmost unused block as the new ``place 1" (head of $k$ new cycles). Proceed to pick images uniformly from the urn until closure as above. In the example there is only one choice for its image (itself), so the new cycles are immediately closed
\begin{equation*}\begin{array}{cccccccccccc}
1&2&3&4&5&6&7&8&9&10&11&12\\
10&11&12&4&6&5&2&1&3&8&9&7
\end{array}
\end{equation*}
and $(1,10,8);(2,11,9,3,12,7);(4);(5,6)$


\end{itemize}



It's clear that the chosen $\sigma$ is uniform in $\Gamma^n \rtimes S_n$. Inspecting the above, one may see that only two pieces of information are relevant to the final cycle lengths. Consider the time of first closure $i$ ($i=3$ in the example). The value of $i$ determines the lengths of the $k$ partial cycles immediately before closing. The cycle type of $\gamma_i \circ \dots \circ \gamma_1$ determines the way in which they close (merging). However since the $\gamma_i$ are all iid uniform permutations this composition is simply a uniform $\Gamma$-permutation. Of course these observations hold for all subsequent closure times as well. Thus if one only cares about the cycle type of $\sigma$ one may throw away all information except the closure times and the cycle type of a single $\Gamma$-permutation for each closure.

Given this, we may directly sample these random variables without sampling the whole permutation. The closure times are 1's in the sequence $\zeta = \zeta_1 \zeta_2\ldots \zeta_n 1$ from \ref{sec42} with $P(\zeta_i = 1) = 1/i$. We must add a terminal one for this spacing interpretation to work, and also choose to write the sequence such that the 1's corresponding to closure times arrive from right to left in the interest of making these sequences infinite later. When $\zeta_i=1$ is the left endpoint of an $l$-space the corresponding $Y_i$ from \ref{sec42} represents the cycle type of $\gamma_i \circ \dots \circ \gamma_{i+l}$, which by inspection yields $a_j(Y_i)$ cycles of length $jl$ as in (\ref{41}). This proves \ref{prop41}. 
\hfill{$\square$}

\section{Three Theorems About Patterns}\label{sec5}

There are myriad theorems about the distribution of various properties of random permutations. Most all can be adapted (and studied) for the Wreath products considered here. For a first stab at the distribution of the longest increasing subsequence, see \cite{ps2024}. The limit of the joint distribution of normalized cycle lengths, and thus statistics of this distribution such as the largest cycle, are studied in \cite{tung2025cutting}. We have not studied the distribution of the limiting shape of a random wreath partition. See \cite{desalvo2019limit} for the classical case of permutations with many variations. This brief section treats two basic statistics: descents and inversions for permutations induced by Wreath products. It also treats a different action of $S_k\times S_n$. For simplicity, we treat throughout the case
\begin{equation*}
S_k^n\rtimes S_n\subseteq S_{kn},
\end{equation*}
but most arguments extend in a transparent way to $\Gamma^n\rtimes S_n$ with $\Gamma\subseteq S_k$.

\subsection{Descents}\label{sec51}

The ``up/down'' pattern in a permutation is captured by the number of descents. For $\sigma\in S_n$, $d(\sigma)=\#\{i:\sigma_i>\sigma_{i+1},1\leq i\leq n-1\}$. If $\sigma$ is uniform in $S_n$,
\begin{equation}
\mu_n=E_n(d(\sigma))=\frac{n-1}2,\qquad\sigma_n^2=\var_n(d(\sigma))=\frac{n+1}{12},
\label{51}
\end{equation}
and, normalized by its mean and standard deviation, $d(\sigma)$ has a limiting standard normal distribution when $n$ is large. Proofs and extensive further discussion are in \cite{pd2017}. The extension to $S_k^n\rtimes S_n$ is straightforward:

\begin{thm}
Let $\sigma$ be uniformly chosen in $S_k^n\rtimes S_n$. Then
\begin{equation*}
E_{kn}(d(\sigma))=\frac{n(k-1)}2+\frac{n-1}2,\qquad\var_{kn}=\frac{n(k+1)}{12}+\frac{n+1}{12},
\end{equation*}
and, normalized by its mean and standard deviation, $d(\sigma)$ has a limiting standard normal distribution when $n$ is large.
\label{thm51}
\end{thm}

\begin{proof}
Write $\sigma=(\gamma_1,\dots,\gamma_n;\eta)$ for $\gamma_i\in S_k$, $\eta\in S_n$. By inspection,
\begin{equation*}
d(\sigma)=\sum_{i=1}^nd(\gamma_i)+d(\eta).
\end{equation*}
All the summands are independent. the first sum has mean $n(k-1)/2=\mu_1$, variance $n(k+1)/12=\sigma_1^2$ and, when normalized, a limiting standard normal distribution. The term $d(\eta)$ has mean $(n-1)/2=\mu_2$, variance $(n+1)/12=\sigma_2^2$ and, when $n$ is large, normalized, a limiting standard normal distribution. This implies
\begin{equation*}
\frac{d(\sigma)-\mu_1-\mu_2}{\sqrt{\sigma_1^2+\sigma_2^2}}=
\frac{\sum_{i=1}^nd(\gamma_i)-\mu_1}{\sigma_1}\frac{\sigma_1}{\sqrt{\sigma_1^2+\sigma_2^2}}+\frac{d(\eta)-\mu_2}{\sigma_2}\frac{\sigma_2}{\sqrt{\sigma_1^2+\sigma_2^2}}
\end{equation*}
has a limiting standard normal distribution.
\end{proof}

The preceding argument works for $k$ fixed \textit{or} growing with $n$.

\subsection{Inversions}\label{sec52}

For $\sigma\in S_n$, the number of inversions in $\sigma$, $I(\sigma)$, is defined as $I(\sigma)=\#\{i<j:\sigma(i)>\sigma(j)\}$. This is a standard measure of disarray, widely used in statistical testing as \textit{Kendal's tau}. Extensive references and further developments are in \cite{pd94}. As is shown there (and classically),
\begin{equation*}
E_n(I(\sigma))=\frac{\binom{n}2}2,\qquad\var_n(I(\sigma))=\frac{n(n-1)(2n+5)}{72},
\end{equation*}
and, normalized by its mean and standard deviation, $I(\sigma)$ has a limiting standard normal distribution. Again, the extension to $S_k^n\rtimes S_n$ is straightforward.

\begin{thm}
Let $\sigma$ be uniformly chosen in $S_k^n\rtimes S_n$. Then
\begin{equation*}
E_{kn}(I(\sigma))=\frac{k^2\binom{n}2}4,\qquad\var_{kn}(I(\sigma))=\frac{k^4n(n-1)(2n+5)}{72}+\frac{nk(k+1)(2k+5)}{72},
\end{equation*}
and, normalized by its mean and standard deviation, $I(\sigma)$ has a limiting standard normal distribution.
\label{thm52}
\end{thm}

\begin{proof}
For $\sigma=(\gamma_1,\dots,\gamma_n;\eta)$, $\gamma_i\in S_k$, $\eta\in S_n$, by inspection
\begin{equation*}
I(\sigma)=k^2I(\eta)+\sum_{i=1}^nI(\gamma_i).
\end{equation*}
The proof now follows as in \ref{thm51}; further details are omitted.
\end{proof}

The proofs here lean on the classical central limit theorem and normal limit theorems for $d(\sigma)$, $I(\sigma)$ in $S_n$. All of these ingredients are available with Berry--Esseen-type errors. These transfer in a straightforward way.

\subsection{A different action}\label{sec53}

John Kingman has pointed to a different appearance of Poisson distributions for the distribution of cycles of large subgroups of $S_N$. Consider $S_k\times S_n$ acting on $[k]\times[n]$ via
\begin{equation*}
(i,j)^{(\sigma,\tau)}=\left(\sigma(i),\tau(j)\right),\qquad\sigma\in S_k,\ \tau \in S_n.
\end{equation*}
The $nk$ points are permuted by $(\sigma,\tau)$ and one may ask about features such as cycles.

\begin{exam}
For $k=4,n=13$, one can picture a deck of cards with A--K of each of the four suits in order, row by row. Permute the rows by $\sigma$ and the columns by $\tau$.
\end{exam}

\begin{exam}
Consider $a_1(\sigma,\tau)$, the number of fixed points. For $(i,j)$ to be fixed, row $i$ and column $j$ must be fixed so that $a_1(\sigma,\tau)=a_1(\sigma)a_1(\tau)$. For $k$ and $n$ large, $a_1(\sigma)$ and $a_1(\tau)$ converge to $XY$ with $X$ and $Y$ independent $\poi(1)$ distributions. The product of independent Poissons is no longer infinitely divisible, so not a compound Poisson variate.
\end{exam}

\begin{exam}
With the preceding notation, for $n$ and $k$ large, \ref{thm53} below gives the joint limiting distribution of $\{a_i\}_{i=1}^{kn}$. As an example,
\begin{equation*}
a_2\Longrightarrow Y_2X_1+(Y_1+2Y_2)X_2,
\end{equation*}
with $X_i,Y_i$ independent $\poi(1/i)$ variates.
\end{exam}

The general form of the limit is a quadratic polynomial in compound Poisson variates. The general result is usefully organized by another result of classical P\'olya theory:

\begin{prop}
\begin{equation*}
Z_{S_k\times S_n}(x_1,\dots,x_{kn})=\sum_{\substack{\lambda\vdash k\\ \mu\vdash n}}\frac1{z_\lambda z_\mu}\prod_{i,j}x_{[i,j]}^{(i,j)a_i(\lambda)a_j(\mu)},
\end{equation*}
where $[i,j],(i,j)$ denote the least common multiple and greatest common divisor, $1\leq i\leq k$, $1\leq j\leq n$. For generalizations of this result to products of subgroups see \cite{wx93}.
\label{prop51}
\end{prop}

\begin{exam}
\begin{equation*}
Z_{S_2\times S_3}(x_1,\dots,x_6)=\frac1{12}\{x_1^6+3x_1^2x_2^2+2x_3^2+4x_2^3+2x_6\}.
\end{equation*}
The general theorem follows from the proposition using Poissonization as in \eqref{26}.
\end{exam}

\begin{thm}
Pick $(\sigma,\tau)$ uniformly in $S_k\times S_n$. Let $a_l(\sigma,\tau)$ be the number of $l$-cycles in the product action. Then, for $k$ and $n$ large,
\begin{equation*}
\{a_l\}_{l=1}^\infty\Longrightarrow\{A_l\}_{l=1}^\infty,
\end{equation*}
with
\begin{equation}
A_l=\sum_{a\mid l}X_a\sum_{i:[i,a]=l}(i,a)Y_i
\label{52}
\end{equation}
and
\begin{equation*}
X_a\sim\poi\left(\frac1{a}\right),\quad Y_i\sim\poi\left(\frac1{i}\right),\quad\text{all $X_a,Y_i$ independent.}
\end{equation*}
\label{thm53}
\end{thm}

\begin{proof}
Using \ref{prop51} and \eqref{26}, the generating function
\begin{equation*}
\sum_{n=0}^\infty Z_{S_k\times S_n}(x_1,\dots,x_{kn})(1-t)t^n=\sum_{\lambda\vdash k}\frac1{z_\lambda}\exp\left\{\sum_{a=1}^\infty\frac{t^a}{a}(s_{a\lambda}-1)\right\},
\end{equation*}
with
\begin{equation*}
s_{a\lambda}=\prod_{i=1}^kx_{[a,i]}^{(a,i)a_i(\lambda)}.
\end{equation*}
Consider the marginal distribution of $a_l$. Its generating function is obtained by setting $x_i=1$ for $i\neq l$ and $x_l=x$. Fix $a$ and $\lambda$ and consider the exponent of $x$ in $s_{a\lambda}$. This is $N_l(a,\lambda)=\sum_{i:[a,i]=l}(i,a)a_i(\lambda)$ \textit{provided} $a\mid l$. It is zero otherwise. Thus
\begin{equation*}
\sum_{n=0}^\infty E_{S_k\times S_n}(x^{a_l})(1-t)t^n=\sum_{\lambda\vdash k}\frac1{z_\lambda}\exp\left\{\sum_{a\mid l}\frac{t^a}{a}x^{N_l(a,\lambda)}\right\}.
\end{equation*}
For fixed $\lambda$ the exponential is the generating function of the compound Poisson distribution
\begin{equation*}
\sum_{a\mid k}X_aN_l(a,\lambda)\qquad\text{with }X_a\sim\poi\left(\frac{t^k}{a}\right)\text{ independent.}
\end{equation*}
As $k$ tends to infinity, under the probability $1/z_\lambda$, the joint distribution of $\{a_i\}_{i=1}^k$ tends to the independent Poisson $\{Y_i\}_{i=1}^\infty$. Now let $t\to1$, the Tauberian arguments of \cite{shepp} complete the proof that the marginal distribution of $a_l$ converges to $A_l$ as in \eqref{52}. The result for the joint distribution is similar and further details are suppressed.
\end{proof}

Coupling as in the proof of \ref{thm41} also gives

\begin{thm}
    For any $b \in \bbn$ and $(\sigma,\tau)$ uniform in $S_{f(n)} \times S_n$, let $\mu$ be the distribution of $\{a_l(\sigma,\tau)\}_{l=1}^b$ and $\nu$ be the distribution of $\{A_l\}_{l=1}^b$. Then
\begin{equation*}
\|\mu-\nu\|_{\tv}\leq \frac{2b}{f(n)} + \frac{2b}{n}.
\end{equation*}
\label{thmprodcoup}
\end{thm}
\begin{proof}[Proof sketch]
    Taking a Bernoulli sequence for each of $\sigma \in S_{f(n)}$ and $\tau \in S_n$ as in \ref{sec42} one has that a space (cycle) of length $i$ in the first and a space of length $j$ in the second combine to create $(i,j)$ cycles of length $[i,j]$ in $(\sigma,\tau)$. The same computation as in the proof of \ref{thm42} and noting the symmetry in the two dimensions then gives the bound.
\end{proof}


\section{Acknowledgments and Funding Sources}

Acknowledgments: We thank David Aldous, Sourav Chatterjee, Jason Fulman, Bob Guralnick, Michael Howes, John Kingman, Gerad Letac, Martin Liebeck, Arun Ram and Chenyang Zhong for their help throughout this project.

Funding: This work was supported by the National Science Foundation [grant number 1954042]. 

  \bibliographystyle{elsarticle-num} 
  \bibliography{Diaconis-refs} 

\end{document}